\documentclass[11pt,a4paper]{article}
\usepackage{amsmath,amsfonts,amssymb,latexsym,graphics,epsfig,url}
\usepackage{color}
\usepackage{amsthm,enumerate}
\usepackage[english]{babel}
\usepackage{graphicx}

\input amssym.def
\newsymbol\rtimes 226F
\newfont{\nset}{msbm10}
\newcommand{\ns}[1]{\mbox{\nset #1}}

\newcommand{\executeiffilenewer}[3]{%
\ifnum\pdfstrcmp{\pdffilemoddate{#1}}%
{\pdffilemoddate{#2}}>0%
{\immediate\write18{#3}}\fi%
}

\newtheorem{theorem}{Theorem}[section]

\newtheorem{corollary}[theorem]{Corollary}
\newtheorem{proposition}[theorem]{Proposition}

\DeclareMathOperator{\dgr}{dgr}
\DeclareMathOperator{\tr}{tr}
\DeclareMathOperator{\spec}{sp}

\def\R{\ns R}

\def\e{\mbox{\boldmath $e$}}
\def\j{\mbox{\boldmath $j$}}

\def\z{\mbox{\boldmath $z$}}

\def\vecu{\mbox{\boldmath $u$}}
\def\vecv{\mbox{\boldmath $v$}}

\def\vecphi{\mbox{\boldmath $\phi$}}
\def\vec0{\mbox{\boldmath $0$}}

\def\A{\mbox{\boldmath $A$}}
\def\B{\mbox{\boldmath $B$}}

\def\D{\mbox{\boldmath $D$}}

\def\G{\Gamma}

\def\I{\mbox{\boldmath $I$}}

\def\M{\mbox{\boldmath $M$}}
\def\O{\mbox{\boldmath $O$}}
\def\S{\mbox{\boldmath $S$}}

\def\I{\mbox{\boldmath $I$}}

\def\e{\mbox{\boldmath $e$}}
\def\f{\mbox{\boldmath $f$}}
\def\j{\mbox{\boldmath $j$}}

\def\z{\mbox{\boldmath $z$}}
\def\vecphi{\mbox{\boldmath $\phi$}}

\def\vec0{\mbox{\bf 0}}

\def\tr{\mathop{\rm tr}\nolimits}

\def\G{\Gamma}

\def\Re{\mathbb R}

\begin{document}

\title{On the $k$-independence number of graphs}
\author{A. Abiad$^a$, G. Coutinho$^b$, M. A. Fiol$^c$ \\
$^{a}${\small Department of Quantitative Economics} \\
{\small Maastricht University, Maastricht, The Netherlands,}\\
{\small {\tt{a.abiadmonge@maastrichtuniversity.nl}}}\\
$^b${\small Department of Computer Science}\\
{\small Federal University of Minas Gerais, Belo Horizonte, Brazil}\\
{\small{\tt {gmcout@gmail.com}}}\\
$^{c}${\small Departament de Matem\`atiques} \\
{\small Universitat Polit\`ecnica de Catalunya, Barcelona, Catalonia} \\
{\small Barcelona Graduate School of Mathematics}\\
{\small{\tt{miguel.angel.fiol@upc.edu}}}}

\maketitle

\begin{abstract}
This paper generalizes and unifies the existing spectral bounds on the $k$-independence number of a graph, which is the maximum size of a set of vertices at pairwise distance greater than $k$. The previous bounds known in the literature follow as a corollary of the main results in this work. We show that for most cases our bounds outperform the previous known bounds. Some infinite graphs where the bounds are tight are also presented. Finally, as a byproduct, we derive some lower spectral bounds for the diameter of a graph.
\end{abstract}

\maketitle

\noindent{\em Keywords:} Graph,  $k$-independence number, Spectrum, Interlacing, Regular partition, Antipodal distance-regular graph, Diameter.

\noindent{\em Mathematics Subject Classifications:} 05C50, 05C69.

\section{Introduction}
Given a graph $G$, let $\alpha_k = \alpha_k(G)$ denote the size of the largest set of vertices such that any two vertices in the set are at distance larger than $k$. This choice of notation is no coincidence, since actually $\alpha_1$ is just the independence number of a graph.
The parameter $\alpha_k(G)$ therefore represents the largest number of vertices which can be $k+1$ spread out in $G$. It is known that determining $\alpha_k$ is NP-Hard in general \cite{KZ1993}.

The $k$-independence number of a graph is directly related to other combinatorial parameters such as the average distance \cite{FH1997}, packing chromatic number \cite{GHHHR2008}, injective chromatic number \cite{HKSS2002}, and strong chromatic index \cite{M2000}. Upper bounds on the $k$-independence number directly give lower bounds on the corresponding distance or packing chromatic number. Regarding it, Alon and Mohar \cite{AM2002} asked for the extremal value of the distance chromatic number for graphs of a given girth and degree.

In this paper we generalize and improve the known spectral upper bounds for the $k$-independence number from \cite{Fiolkindep} and \cite{act16}. For some cases, we also show that our bounds are sharp.


As far as we are aware, there seems to be some conflict in the existing literature regarding the use of the term `$k$-independence number'. The following list contains the three conflicting definitions, which all, nonetheless, are a natural generalization of the concept of independence number.

\begin{enumerate}
\item
Caro and Hansberg \cite{CaroHansbergkIndep} use the term `$k$-independence number' to denote the maximum size of a set of vertices in a graph whose induced subgraph has maximum degree $k$. Thus, $\alpha_0$ is the usual independence number.
\item
\v{S}pacapan \cite{Spacapankindep}  uses `$k$-independence number' to denote the size of the largest $k$-colourable subgraph of $G$. With this notation, $\alpha_1$ stands for the usual $k$-independence number of $G$.
\item
Fiol \cite{Fiolkindep} and Abiad, Tait, and Cioab\u{a}~\cite{act16} use `$k$-independence number' to denote the size of the largest set of vertices such that any two vertices in the set are at distance larger than $k$.
\end{enumerate}
The latter definition is the one we use in this work.

The first known spectral bound for the independence number $\alpha$ is due to Cvetkovi\'c \cite{c71}.

\begin{theorem}[Cvetkovi\'c \cite{c71}]
\label{thm:cvetkovic}
Let $G$ be a graph with eigenvalues $\lambda_1\ge \cdots \ge \lambda_n$. Then,
$$
\alpha\le \min \{|\{i : \lambda_i\ge 0\}| \quad \text{and} \quad |\{i : \lambda_i\le 0\}|\}.
$$
\end{theorem}

Another well-known result is the following bound due to Hoffman (unpublished; see for instance Haemers  \cite{h95}).

\begin{theorem}[Hoffman \cite{h95}]
\label{thm:hoffman}
If $G$ is a regular graph on $n$ vertices with eigenvalues $\lambda_1\ge \cdots \ge \lambda_n$, then
\[
\alpha \leq n\frac{-\lambda_n }{\lambda_1 - \lambda_n}.
\]
\end{theorem}

Regarding the $k$-independence number, the following three results are known. The first is due to Fiol \cite{Fiolkindep} and requires a preliminary definition. Let $G$ be a graph with distinct eigenvalues $\theta_0 > \cdots > \theta_d$. Let $P_k(x)$ be chosen among all polynomials $p(x) \in \Re_k(x)$, that is, polynomials of real coefficients and degree at most $k$, satisfying $|p(\theta_i)| \leq 1$ for all $i = 1,...,d$, and such that $P_k(\theta_0)$ is maximized. The polynomial $P_k(x)$ defined above is called the {\em $k$-alternating polynomial} of $G$ and  was shown to be unique in \cite{fgy96}, where it was used to study the relationship between the spectrum of a graph and its diameter.

\begin{theorem}[Fiol \cite{Fiolkindep}] \label{thm:fiol}
Let $G$ be a $d$-regular graph on $n$ vertices, with distinct eigenvalues $\theta_0 >\cdots > \theta_d$ and let $P_k(x)$ be its $k$-alternating polynomial. Then,
\[
\alpha_k \leq \frac{2n}{P_k(\theta_0) + 1}.
\]
\end{theorem}

The second and third bounds are due to Abiad, Cioab\u{a}, and Tait \cite{act16}. The first is a Cvetkovi\'c-like approach, whereas the second resembles Hoffman's.

\begin{theorem}[Abiad, Cioab\u{a}, Tait \cite{act16}]
\label{previous1act}
Let $G$ be a graph on $n$ vertices with adjacency matrix $\A$, with eigenvalues $\lambda_1 \geq \cdots \geq \lambda_n$. Let $w_k$ and $W_k$ be respectively the smallest and the largest diagonal entries of $\A^k$. Then,
\[
\alpha_k \leq |\{i : \lambda_i^k \geq w_k(G)\}| \quad \text{and} \quad \alpha_k \leq |\{i : \lambda_i^k \leq W_k(G)\}|.
\]
\end{theorem}

\begin{theorem}[Abiad, Cioab\u{a}, Tait \cite{act16}]\label{previous2act}
\label{thm:abiad}
Let $G$ be a $\delta$-regular graph on $n$ vertices with adjacency matrix $\A$, whose distinct eigenvalues are $\theta_0(=\delta) > \cdots> \theta_d$. Let $\widetilde{W_k}$ be the largest diagonal entry of $\A+\A^2+\cdots+\A^k$. Let $\theta = \max\{|\theta_1| , |\theta_d|\}$. Then,
\[
\label{aida}
\alpha_k \leq n \frac{\widetilde{W_k}+ \sum_{j = 1}^k \theta^j}{\sum_{j = 1}^k \delta^j + \sum_{j = 1}^k\theta^j}.
\]
\end{theorem}

\section{Preliminaries}
For basic notation and results see \cite{biggs,g93}. Let $G=(V,E)$ be a graph with $n=|V|$ vertices, $m=|E|$ edges, and adjacency matrix $\A$ with  spectrum $\spec G=\{\theta_0>\theta_1^{m_1}>\cdots>\theta_d^{m_d}\}$. When the eigenvalues are presented with possible repetitions, we shall indicate them by $\lambda_1 \geq \lambda_2 \geq \cdots\geq \lambda_n$.
Let us consider the scalar product in $\Re_d[x]$:
$$
\langle f,g\rangle_G=\frac{1}{n}\tr(f(\A)g(\A))=\frac{1}{n}\sum_{i=0}^{d} m_i f(\theta_i)g(\theta_i).
$$
The so-called {\em predistance polynomials} $p_0(=1),p_1,\ldots, p_d$ are a sequence of orthogonal polynomials with respect to the above product, with  $\dgr p_i=i$, and are normalized in such a way that $\|p_i\|_G^2=p_i(\theta_0)$ (this makes sense since it is known that $p_i(\theta_0)>0$) for $i=0,\ldots,d$. Therefore they are uniquely determined, for instance, following the Gram-Schmidt process. They were introduced by Fiol and Garriga in \cite{fg97} to prove the so-called `spectral excess theorem' for distance-regular graphs. We also use the sum polynomials $q_i=p_0+\cdots+p_i$, for $i=0,\ldots, d-1$, which are also a sequence of orthogonal polynomials, now with respect to the scalar product
$$
\langle f,g\rangle_{[G]}=\frac{1}{n}\sum_{i=0}^{d-1} m_i(\theta_0-\theta_i) f(\theta_i)g(\theta_i),
$$
and satisfy $1=q_0(\theta_0)<q_1(\theta_0)<\cdots<q_{d-1}(\theta_0)<n$.
 See \cite{cffg09} for further details and applications.

Eigenvalue interlacing is a powerful and old technique that has found countless applications in combinatorics and other fields. This technique will be used in several of our proofs. For more details, historical remarks and other applications see Fiol and Haemers \cite{f99,h95}.

Given square matrices $\A$ and $\B$ with respective eigenvalues $\lambda_1\geq \cdots \geq \lambda_n$ and $\mu_1 \geq \cdots \geq \mu_m$, with $m<n$, we say that the second sequence {\em interlaces} the first if, for all $i = 1,\ldots,m$, it follows that
\[
\lambda_i \geq \mu_i \geq \lambda_{n-m+i}.
\]

\begin{theorem}[Interlacing \cite{f99,h95}]
\label{theo-interlacing}
Let $\S$ be a real $n \times m$ matrix such that $\S^T \S = \I$, and let $\A$ be a $n \times n$ matrix with eigenvalues $\lambda_1 \geq \cdots \geq \lambda_n$. Define $\B = \S^T \A \S$, and call its eigenvalues $\mu_1 \geq\cdots \geq \mu_m$. Then,
\begin{enumerate}[(i)]
\item
The eigenvalues of $\B$ interlace those of $\A$.
\item
If $\mu_i = \lambda_i$ or $\mu_i = \lambda_{n-m+i}$, then there is an eigenvector $\vecv$ of $\B$ for $\mu_i$ such that $\S \vecv$ is eigenvector of $\A$ for $\mu_i$.
\item
If there is an integer $k \in \{0,\ldots,m\}$ such that $\lambda_i = \mu_i$ for $1 \leq i \leq k$,  and $\mu_i = \lambda_{n-m+i}$ for $ k+1 \leq i \leq m$ $(${\em tight interlacing}$)$,  then $\S \B = \A \S$.
\end{enumerate}
\end{theorem}

Two interesting particular cases where interlacing occurs (obtained by choosing
appropriately the matrix $\S$) are the following. Let $\A$ be the adjacency matrix of a graph $G=(V,E)$. First, if $\B$ is a principal submatrix of $\A$, then $\B$ corresponds to the adjacency matrix of an induced subgraph
$G'$ of $G$. Second, when, for a given partition of the vertices of $\G$, say $V=U_1\cup\cdots\cup U_m$, $\B$ is the so-called {\em quotient matrix} of $\A$, with elements
$b_{ij}$, $i,j=1,\ldots,m$, being the average row sums of the corresponding block $\A_{ij}$ of $\A$. Actually, the quotient matrix $\B$ does not need to be
symmetric or equal to $\S^\top\A\S$, but in this case $\B$ is
similar to---and therefore has the same spectrum as---
$\S^\top\A\S$.
In the second case, if the interlacing is tight,
Theorem~\ref{theo-interlacing}$(iii)$ reflects that $\S$ corresponds to a {\em
regular} (or {\em equitable}) partition of $\A$, that is, each
block of the partition has constant row and column sums. Then
the bipartite induced subgraphs $G_{ij}$, with adjacency matrices $\A_{ij}$, for $i\neq i$, are biregular, and the subgraphs
$G_{ii}$ are regular.

We finally recall that the {\it Kronecker product} of two matrices
$\A=(a_{ij})$ and $\B$, denoted by $\A\otimes \B$, is obtained by
replacing each entry $a_{ij}$  with the  matrix
$a_{ij}\B$, for all $i$ and $j$. Then, if $\vecu$ and $\vecv$ are eigenvectors
of $\A$ and $\B$, with corresponding eigenvalues $\lambda$ and
$\mu$, respectively, then $\vecu\otimes \vecv$ (seeing $\vecu$ and $\vecv$  as matrices) is an eigenvector of $\A\otimes \B$, with eigenvalue
$\lambda\mu$.

\section{Three main results}

The  objective  of  this  section  is  to  obtain  three general spectral upper  bounds  for  $\alpha_k$. Our first Theorem \ref{theo1} is  a  very  general  bound. Since it depends on a certain polynomial $p\in \Re_k[x]$, it is difficult to study when it is sharp in general, but it can be seen as a generalization on the previous Theorem \ref{previous1act}. Our second Theorem \ref{theo2} is a significant improvement to Theorem \ref{previous2act} and is sharp for some values of $k$, as shown using computer-assisted calculations. Finally, our last Theorem \ref{theo3} provides an antipodal-like bound that generalizes Theorem \ref{thm:fiol}.

Let $G$ be a graph with eigenvalues $\lambda_1\ge \lambda_2\ge \cdots\ge \lambda_n$. Let $[2,n]=\{2,3,\ldots,n\}$. Given a polynomial $p\in \Re_k[x]$, we define the following parameters:
\begin{itemize}
\item
$W(p) = \max_{u\in V}\{(p(\A))_{uu}\}$;
\item
$w(p) = \min_{u\in V}\{(p(\A))_{uu}\}$;
\item
$\Lambda(p) = \max_{i\in[2,n]}\{p(\lambda_i)\}$;
\item
$\lambda(p) = \min_{i\in[2,n]}\{p(\lambda_i)\}$.
\end{itemize}

In the following three results, $G$ is a graph with $n$ vertices, adjacency matrix $\A$ and eigenvalues $\lambda_1\ge \lambda_2\ge \cdots\ge \lambda_n$. Let  $p\in \Re_k[x]$ with corresponding parameters $W(p)$, $w(p)$, $\Lambda(p)$ and $\lambda(p)$.

\subsection{A Cvetkovi\'c-like bound}

\begin{theorem}
\label{theo1}
Let  $p\in \Re_k[x]$ with corresponding parameters $W(p)$, $w(p)$. Then, the $k$-independence number of $G$ satisfies the bound
$$
\alpha_k\le \min \{|\{i : p(\lambda_i)\ge w(p)\}|, \ |\{i : p(\lambda_i)\le W(p)\}|.
$$
\end{theorem}
\begin{proof}
We use the interlacing approach. Assume $U$ is a $k$-independent set of $G$. We arrange the columns and rows of $\A$ to have the vertices of $U$ appearing in the first positions. This implies that, for any polynomial $p(x)$ of degree at most $k$, the principal submatrix with the first $|U|$ rows and columns of $p(\A)$ is diagonal. Call this matrix $\D$.  Choosing $\S^T = \left(\begin{array}{c|c} \I_k & \O \end{array}\right)$, we have
\[\S^T p(\A) \S = \D.\]
Let $\mu$ be the smallest eigenvalue of $\D$. From interlacing, it follows that there must be at least $|U|$ eigenvalues of $p(\A)$ larger than $\mu$. Noting that $w(p) \leq \mu$, we have $|U| \leq | \{ i : p(\lambda_i) \geq w(p)\}|$. The other bound is proved analogously.
\end{proof}

It is well known that Theorem \ref{thm:cvetkovic} (Cvetkovic's bound) holds for weighted adjacency matrices. Thus, in our result above, instead of talking about polynomials of degree at most $k$, we could simply say ``let $\M$ be any matrix whose support consists of entries corresponding to vertices at distance at most $k$\ldots". The downside of this approach is that it is in general quite hard to find the optimal $\M$. Our approach in this work is interesting if one can come up with a good choice for the polynomial $p\in \Re_k[x]$ or with an efficient method (like linear programming) to compute it in practice.

An analogous remark also applies to the next results, if one considers that the $k$-independence number of a graph $G$ is precisely the independence number of the graph formed by making all pairs of vertices of $G$ at distance at most $k$ adjacent. For this graph, say $G^{(k)}$, one can formulate an optimization problem over completely positive matrices whose optimal value is equal to its independence number \cite{klerkpasech}. The semidefinite relaxation of this programming yields the Lov\'{a}sz Theta number of $G^{(k)}$, which upper bounds $\alpha_k(G)$. The spectral bounds we find below can all be obtained as the objective value of some feasible solution to the minimization formulation of the Lov\'asz Theta semidefinite programming, therefore they are all larger or equal than the Lov\'{a}sz Theta number of $G^{(k)}$. We point however that computing our spectral bounds is significantly faster than solving an SDP, and in many cases they perform fairly good, as we will point in some tables below.

\subsection{A Hoffman-like bound}

\begin{theorem}
\label{theo2}
Let $G$ be a regular graph with $n$ vertices and eigenvalues $\lambda_1\ge\cdots\ge \lambda_n$.
Let  $p\in \Re_k[x]$ with corresponding parameters $W(p)$ and $\lambda(p)$, and assume $p(\lambda_1) > \lambda(p)$. Then,
\begin{equation}
\label{eq:thm2}
\alpha_k\le n\frac{W(p)-\lambda(p)}{p(\lambda_1)-\lambda(p)}.
\end{equation}
\end{theorem}

\begin{proof}
Let $\A$ be the adjacency matrix of $G$. Let $U$ be a $k$-independent set of $G$ with $r=|U|=\alpha_k(G)$ vertices. Again, assume the first columns (and rows) of $\A$ correspond to the vertices in $U$. Consider the partition of said columns according to $U$ and its complement. Let $\S$ be the normalized characteristic matrix of this partition. The quotient matrix of $p(\A)$ with regards to this partition is given by
\begin{equation}
\label{B_k=2}
\S^T p(\A) \S = \B_k=
\left(
\begin{array}{cc}
\frac{1}{r}\sum_{u\in U}(p(\A))_{uu} & p(\lambda_1)-\frac{1}{r}\sum_{u\in U}(p(\A))_{uu}\\
\frac{r p(\lambda_1)-\sum_{u\in U}(p(\A))_{uu}}{n-r}  & p(\lambda_1)-\frac{r p(\lambda_1)-\sum_{u\in U}(p(\A))_{uu}}{n-r}
\end{array}
\right),
\end{equation}
with eigenvalues $\mu_1=p(\lambda_1)$ and
$$
\mu_2=\tr \B_k-p(\lambda_1)=\frac{1}{r}\sum_{u\in U}(p(\A))_{uu}-\frac{r p(\lambda_1)-\sum_{u\in U}(p(\A))_{uu}}{n-r}.
$$
Then, by interlacing, we have
\begin{equation}
\label{interlacing:theo2}
\lambda(p)\le \mu_2\le W(p)-\frac{r p(\lambda_1)-rW(p)}{n-r},
\end{equation}
whence, solving for $r$ and taking into account that $p(\lambda_1)-\lambda(p)>0$, the result follows.
\end{proof}

Let us now consider some particular cases of  Theorem \ref{theo2}.

\subsubsection*{The case $k=1$.}
As mentioned above, $\alpha_1$ coincides with the standard independence number. In this case we can take $p$ as any linear polynomial satisfying $p(\lambda_1)>\lambda(p)$, say $p(x)=x$. Then, we have $W(p)=0$, $p(\lambda_1)=\lambda_1$, $\lambda(p)=p(\lambda_n)=\lambda_n$, and \eqref{eq:thm2} gives
\begin{equation}
\label{eq-k=1}
\alpha_1=\alpha\le n\frac{-\lambda_n}{\lambda_1-\lambda_n},
\end{equation}
which is Hoffman's bound in Theorem \ref{thm:hoffman}.

\subsubsection*{The case $k=2$.}
By making the right choice of a polynomial of degree two, we get the following result.

\begin{corollary}
\label{k=2}
Let $G$ be a $\delta$-regular graph with $n$ vertices, adjacency matrix $\A$,  and distinct eigenvalues $\theta_0(=\delta)>\theta_1> \cdots > \theta_d$ with $d\ge 2$. Let $\theta_i$ be the largest eigenvalue such that $\theta_i\le -1$. Then, the $2$-independence number satisfies
\begin{equation}
\label{eq-k=2}
\alpha_2\le n\frac{\theta_0+\theta_i\theta_{i-1}}{(\theta_0-\theta_i)
(\theta_0-\theta_{i-1})}.
\end{equation}
If the bound is attained, the matrix $\A^2-(\theta_i+\theta_{i-1})\A$ has a regular partition  $($with a set of $\alpha_2$ $2$-independent vertices and its complement$)$ with quotient matrix
\begin{equation}
\label{quotient-k=2}
\B=
\left(
\begin{array}{cc}
\delta & \delta^2-(\theta_i+\theta_{i-1}+1)k\\
 \delta+\theta_i\theta_{i-1} & \delta^2-(\theta_i+\theta_{i-1}+1)k-\theta_i\theta_{i-1}
\end{array}
\right).
\end{equation}
Moreover, this is the best possible bound that can be obtained by choosing a polynomial and applying Theorem \ref{theo2}.
\end{corollary}
\begin{proof}
Note that only the last assertion is non-trivial, in view of Theorem \ref{theo2}. We now show why it holds. Let $p(x)=ax^2+bx+c$ and suppose first that $a>0$. Then, from the expression of the bound in \eqref{eq:thm2}, there is no loss of generality if we take $a=1$ and $c=0$.
Then, the minimum of the polynomial $p(x)=x^2+bx$ is attained at $x=-b/2$ and, hence,
given $b$, the minimum $\lambda(p)$ must be equal to $p(\theta_i)$ where $\theta_i$ is the eigenvalue closest to $-b/2$. Thus, from $(\theta_i+\theta_{i+1})/2\le -b/2\le (\theta_i+\theta_{i-1})/2$ we can write that $b=-\theta_i+\tau$ for $\tau\in[-\theta_{i-1},-\theta_{i+1}]$. Then, with $W(p)=\theta_0$, $\lambda(p)=p(\theta_i)=\tau\theta_i$, and $p(\theta_0)=\theta_0^2+(-\theta_i+\tau)\theta_0\ (>\lambda(p))$, the bound in \eqref{eq:thm2}, as a function of $\tau$, is
$$
\Phi(\tau)=n\frac{\theta_0-\theta_i\tau}{(\theta_0-\theta_i)(\theta_0+
\tau)},
$$
with derivative $\Phi'(\tau)=n\frac{-\theta_0(1+\theta_i)}{(\theta_0-\theta_i)(\theta_0+
\tau)^2}$. Consequently, the resulting bound $\Phi(\tau)$ is an increasing, constant, or decreasing function depending on $\theta_i<-1$, $\theta_i=-1$, or $\theta_i>-1$, respectively. Since we are interested in the minimum value of $\Phi$, we reason as follows:
\begin{itemize}
\item
If $\theta_i<-1$, we must take the value of $\tau$ as small as possible, that is $\tau=-\theta_{i-1}$, which gives
$\alpha_2\le \Phi(-\theta_{i-1})=n\frac{\theta_0+\theta_i\theta_{i-1}}
{(\theta_0-\theta_i)(\theta_i-\theta_{i-1})}$. Moreover, iterating the reasoning, we eventually take for $\theta_i$ the largest eigenvalue smaller than $-1$, as claimed.
\item
If $\theta_i=-1$, we have that $\theta_{i+1}>-1$ and, with $\theta_i$ taking the role of $\theta_{i+1}$, we are in the next case.
\item
If $\theta_i>-1$, we must take the value of $\tau$ as large as possible, that is $\tau=-\theta_{i+1}$, which gives
$\alpha_2\le \Phi(-\theta_{i+1})=n\frac{\theta_0+\theta_i\theta_{i+1}}
{(\theta_0-\theta_i)(\theta_i-\theta_{i+1})}$. Again, iterating the procedure, we eventually take for $\theta_i$ the smallest eigenvalue greater than $-1$, as claimed. Moreover, $\theta_{i+1}$ is the largest eigenvalue that is as most $-1$, in agreement with our claim.
\end{itemize}

To show that our choice of the polynomial $p$ is best possible, we assume now that $a<0$ and, then, we reason with $p(x)=-x^2+bx$. First, to satisfy the condition $p(\theta_0)>\lambda(p)$, we must have $b>\theta_0+\theta_d$. Then, $\lambda(p)=p(\theta_d)=-\theta_d^2-b\theta_d$ and the bound in \eqref{eq:thm2} as a function of $b$, is
$$
\Phi(b)
=n\frac{-\theta_0+\theta_d^2-b\theta_d}{-\theta_0^2+\theta_d^2+b(\theta_0-\theta_d)},
$$
which is decreasing for $b>\theta_0+\theta_d$. Then, we should take $\lim_{b\rightarrow\infty}\Phi(b)=n\frac{-\theta_d}{\theta_0-\theta_d}$. But this is again the Hoffman's bound in \eqref{eq-k=1} for $\alpha_1$, which is trivial for $\alpha_2$.

If equality in \eqref{eq-k=2} holds, from \eqref{interlacing:theo2} we conclude that $\mu_2=\lambda(p)$ and, since $\mu_1=p(\lambda_1)(=\Lambda(p))$, the interlacing is tight and the partition of $p(\A)$ is regular. Finally,  its  quotient matrix $\B$ in \eqref{quotient-k=2} is obtained from  \eqref{B_k=2} by using the right polynomial $p(x)$ and the bound of $\alpha_2$ in \eqref{eq-k=2}.
\end{proof}

Before giving some examples, we notice that the above choice of $\theta_i(\le -1)$ always make sense because it is easy to prove (for example, using interlacing) that the smallest eigenvalue of a graph always satisfies this condition.

In Table \ref{table1} we show the results of testing all named graphs from SAGE. The performance of our purely spectral bound from Corollary \ref{k=2} (column denoted ``Corollary \ref{k=2}") is compared to the best bound that appears in \cite{act16} (column denoted ``Bound \cite{act16}"), which, to our knowledge, is the best known bound for $\alpha_2$ that can be obtained via spectral methods only. Moreover, we compare the mentioned bounds to the values of the floor of the Lov\'asz theta number of the distance at most 2 graph (column denoted ``$\Theta_2$ \cite{L1979}"). The last column of the following table provides the actual value of $\alpha_2$. Regarding the last column, entries that say ``time" denote that the computation took longer than 60 seconds on a standard laptop. The parameter $\alpha_k$ is computationally hard to determine, and it is not clear how long it would take to calculate the table entries that timed out. Note that in almost all cases our bound from Corollary \ref{k=2} performs significantly better than the best known spectral bound.

\begin{table}
\scriptsize{
\begin{center}
\begin{tabular}{|lcccc|}
\hline
Name & Bound in \cite{act16} & $\Theta_2$ \cite{L1979} & Corollary \ref{k=2} & $\alpha_2$\\
\hline
Balaban 10-cage & $32$ & $17$ & $17$ & $17$ \\
Frucht graph & $6$ &  $3$ & $3$ & $3$ \\
Meredith graph & $20$  & $10$ & $14$ & $10$ \\
Moebius-Kantor graph & $8$  & $4$ & $4$ & $4$ \\
Bidiakis cube & $5$ &$2$ & $3$ & $2$ \\
Gosset graph & $2$ &  $2$ & $2$ & $2$ \\
Balaban 11-cage & $41$ & $26$ & $27$ & time \\
Gray graph & $33$ &  $11$ & $11$ & $11$ \\
Nauru graph & $10$ & $6$ & $6$ & $6$ \\
Blanusa first snark graph & $8$ &  $4$ & $4$ & $4$ \\
Pappus graph & $9$  & $3$ & $3$ &$3$ \\
Blanusa second snark graph & $8$ & $4$ & $4$ &$4$ \\
Brinkmann graph & $6$ & $3$ & $3$ &$3$ \\
Harborth graph & $24$ & $10$ & $10$ &$10$ \\
Perkel graph & $12$ &  $5$ & $5$ & $5$ \\
Harries graph & $32$ & $17$ & $17$ &$17$ \\
Bucky ball & $23$ &  $12$ & $14$ &$12$ \\
Harries-Wong graph & $32$ & $17$ & $17$ &$17$ \\
Robertson graph & $4$ & $3$ & $3$ &$3$ \\
Heawood graph & $2$ &  $2$ & $2$ &$2$ \\
Cell 600 & $18$  & $8$ & $8$ & $8$ \\
Cell 120 & $302$ & $120$ & $120$ & $120$ \\
Hoffman graph & $6$ & $2$ & $2$ &$2$ \\
Sylvester graph & $8$ & $6$ & $6$ &$6$ \\
Coxeter graph & $13$ & $7$ & $7$ &$7$ \\
Holt graph & $10$ & $3$ & $4$ & $3$ \\
Szekeres snark graph & $25$ & $10$ & $12$ & $9$ \\
Desargues graph & $10$ & $5$ & $5$ & $4$ \\
Horton graph & $50$ &  $24$ & $24$ & $24$ \\
Dejter graph & $44$ &  $16$ & $16$ & $16$ \\
Tietze graph & $5$ & $3$ & $3$ & $3$ \\
Double star snark & $12$ & $7$ & $7$ & $6$ \\
Truncated icosidodecahedron & $60$ & $28$ & $30$ & $26$ \\
Durer graph & $5$ & $2$ & $2$ & $2$ \\
Klein 3-regular Graph & $22$ & $13$ & $13$ & $12$ \\
Truncated tetrahedron & $5$ & $3$ & $3$ & $3$ \\
Dyck graph & $14$ & $8$ & $8$ & $8$ \\
Klein 7-regular graph & $3$ & $3$ & $3$ & $3$ \\
Tutte 12-cage & $44$ & $28$ & $28$ & time \\
Ellingham-Horton 54-graph & $32$ & $12$ & $13$ & $11$ \\
Tutte-Coxeter graph & $10$ & $6$ & $6$ & $6$ \\
Ellingham-Horton 78-graph & $38$ & $19$ & $19$ & $18$ \\
Ljubljana graph & $44$ & $27$ & $27$ & time \\
Tutte graph & $21$ & $10$ & $11$ & $10$ \\
F26A graph & $12$ & $6$ & $6$ & $6$ \\
Watkins snark graph & $25$ & $9$ & $12$ & $9$ \\
Flower snark & $7$ & $5$ & $5$ & $5$ \\
Markstroem graph & $11$ & $6$ & $6$ & $6$ \\
Wells graph & $6$ & $3$ & $3$ & $2$ \\
Folkman graph & $10$ & $3$ & $3$ & $3$ \\
Foster graph & $44$ & $22$ & $22$ & $21$ \\
McGee graph & $10$ & $5$ & $6$ & $5$ \\
Franklin graph & $6$ & $2$ & $3$ & $2$ \\
Hexahedron & $2$ & $2$ & $2$ & $2$ \\
Dodecahedron & $9$ & $4$ & $4$ & $4$ \\
Icosahedron & $2$ & $2$ & $2$ & $2$ \\
\hline
\end{tabular}
\end{center}}
\caption{Comparison between different bounds for the $2$-independence number.}
\label{table1}
\end{table}
\normalsize

Apart from the examples in the table, we describe next two infinite families of (distance-regular) graphs where the bound of Corollary \ref{k=2} is tight.

First, suppose that $G$ is a connected {\em strongly regular} graph with parameters $(n,k,a,c)$ (here we follow the notation of Godsil \cite{g93}). That is, $G$ is a $k$-regular graph with $n$ vertices, such that every pair of adjacent vertices have $a$ common neighbours, and every pair of non-adjacent vertices have $c>0$ common neighbours. Then, $G$ has distinct eigenvalues
$$
\textstyle
\theta_0=k,\quad \theta_1=\frac{1}{2}[a-c+\sqrt{(a-c)^2+4(a-c)}],\quad
\theta_2=\frac{1}{2}[a-c-\sqrt{(a-c)^2+4(a-c)}],
$$
(for instance, see again \cite{g93}).
Moreover, as $n=1+k+\frac{1}{c}[k(k-a+1)]$, Corollary  \ref{k=2} gives
$\alpha_2=1$, as it should be since $G$ has diameter $2$.

Now, let us take $G$ an antipodal and bipartite distance-regular graph, with degree $k$ and diameter $3$. As shown in \cite{bcn89}, these graph have $n=2(k+1)$ vertices, intersection array $\{k,k-1,1;1,k-1,k\}$, and distinct eigenvalues
\begin{equation}
\label{eigenvals-LK(2,k+1)}
\theta_0=k,\ \theta_1=1,\ \theta_2=-1, \ \theta_3=-k.
\end{equation}
They are also uniquely determined for each $k$. They are the complement of the line graph of the complete bipartite graph $K_{2,k+1}$, denoted by $L(K_{2,k+1})$. Alternatively, $G$ can be constructed from $K_{k+1,k+1}$ minus a perfect matching. In particular for $k=2,3$ we obtain the hexagon and the $3$-cube, respectively. In Figure \ref{fig1} we shown the case of $k=5$.
With the eigenvalues in \eqref{eigenvals-LK(2,k+1)}, Corollary \ref{k=2} then gives $\alpha_2\le 2$, which is tight since the graph is $2$-antipodal, as shown in the example of the figure. Moreover, since $\theta_1+\theta_2=0$, the polynomial in Corollary \ref{k=2} is just $p(x)=x^2$, and hence the matrix $\A^2$ has
a regular partition with the following quotient matrix given by \eqref{quotient-k=2}:
\begin{equation}
\label{quotient-k=2-example}
\B=
\left(
\begin{array}{cc}
k & k(k-1)\\
k-1 & k(k-1)+1
\end{array}
\right).
\end{equation}

\begin{figure}[t]
\begin{center}
\includegraphics[width=8cm]{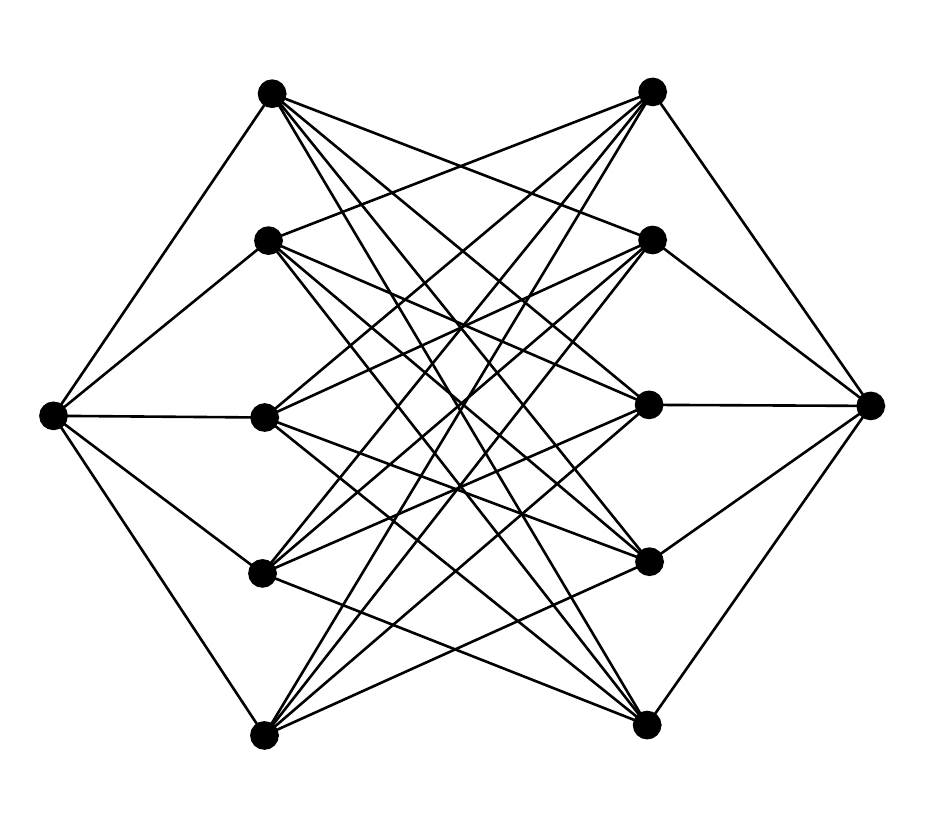}
\caption{An antipodal and bipartite distance-regular graph with degree $5$.}
\label{fig1}
\end{center}
\end{figure}

\subsubsection*{The case of general $k$.}
Another consequence of Theorem \ref{theo2} is the following Corollary \ref{coro-gen-k}, which is closely related to Theorem \ref{thm:abiad}. This is due to the fact that both results make use of the same polynomial\\  ${p(x)=x+x^2+\cdots+x^k}$, although the bounds given in Corollary \ref{coro-gen-k} constitute a significant improvement.

Regarding the next result, first note that for any $p(x)$, if $W(p) > p(\lambda_1)$, then the bound in Theorem \ref{theo2} is trivial. If $p(\lambda_1) \geq W(p)$, then any positive constant can be added to both the numerator and denominator of the quotient in Theorem \ref{theo2} without changing the sign of the inequality. In particular, if $p(\lambda_i) \geq 0$ for all $i$, we can choose to ignore the term $\lambda(p)$ in the bound. On the other hand, given $p(x)$, one can always define the polynomial $q(x) = p(x) - \lambda(p)$, which satisfies $q(\lambda_i) \geq 0$ for all $i$. It is therefore not hard to see that the following corollary is equivalent to Theorem \ref{theo2}, in the sense that the minimization of the ratio over all polynomials satisfying the hypotheses will yield to the same bound.

\begin{corollary}\label{cor:pos}
Let $G$ be a regular graph with $n$ vertices and eigenvalues $\lambda_1\ge\cdots\ge \lambda_n$.
Let  $p\in \Re_k[x]$ with corresponding parameter $W(p)$ and so that $p(\lambda_i) \geq 0$ for all $i$. Then
\begin{equation}
\alpha_k\le n\frac{W(p)}{p(\lambda_1)}.
\end{equation}
\end{corollary}

If $\nu  = \max\{|\lambda_2| , |\lambda_n|\}$, and upon choosing $p(x) = \sum_{\ell = 1}^k x^\ell + \sum_{\ell = 1}^k \nu^\ell$, it is easy to see that $p(\lambda_i) \geq 0$ for all $i$, and that the previous corollary gives precisely Theorem \ref{thm:abiad}. We can do better using the same polynomial, noting that $\lambda(p)$ can be computed explicitly for the case when $k$ is odd, and a reasonable lower bound for it can be found for when $k$ is even.

\begin{corollary}
\label{coro-gen-k}
Let $G$ be a $\delta$-regular graph with $n$ vertices and distinct eigenvalues $\theta_0(=\delta)>\theta_1> \cdots > \theta_d$. Let $W_k=W(p)=\max_{u\in V}\{\sum_{i=1}^k(\A^k)_{uu}\}$.
Then, the $k$-independence number of $G$ satisfies the following:
\begin{itemize}
\item[$(i)$]
If $k$ is odd, then
\begin{equation}
\label{eq-k}
\alpha_k\le n\frac{W_k-\sum_{j=0}^k \theta_d^j}{\sum_{j=0}^k \delta^j-\sum_{j=0}^k \theta_d^j}.
\end{equation}
\item[$(ii)$]
If $k$ is even, then
\begin{equation}
\alpha_k\le n\frac{W_k+1/2}{\sum_{j=0}^k \delta^j+1/2}.
\end{equation}
\end{itemize}
\end{corollary}
\begin{proof}
For odd $k$, the polynomial $p(x)=x+x^2+\cdots+x^k$ is strictly increasing for any $x$, thus the (negative) value of $\lambda(p)$ is always  $\sum_{j=0}^k \theta_d^j$, and Theorem \ref{theo2} gives the desired bound in (i).

For even $k$, the polynomial $p(x)$ is negative precisely between $-1$ and $0$, and its minimum is bounded below by $-(1/2)$. In fact, it approaches $-(1/2)$ as $k$ grows. Therefore (ii) follows from Corollary \ref{cor:pos} applied to $p(x) + 1/2$.
\end{proof}

\subsubsection*{The case of walk-regularity.}
Assume now that $G$ is {\em walk-regular}, that is, for any fixed $k\ge 0$, the number $a_{uu}^{(k)}$ of closed walks of length $k$ rooted at a vertex $u$ does not depend on $u$. As a consequence, for any polynomial $p(x)$, $W(p) = \frac{1}{n} \tr p(\A)$.
\begin{corollary}
\label{coro-walk-reg}
Let $G=(V,E)$ be a walk-regular graph, with degree $\delta$, $n$ vertices, and spectrum $\spec G=\{\theta_0(=\delta),\theta_1^{m_1},\ldots,\theta_d^{m_d}\}$, where $\theta_0>\cdots>\theta_d$. Let $p=q_k$ be the sum polynomial $q_k=p_0+\cdots+p_k$, for $k>0$, where the $p_i$'s stand for the predistance polynomials of $G$.
Then the $k$-independence number of $G$ satisfies
$$
\alpha_k(G)\le n\frac{1-\lambda(q_k)}{q_k(\delta)-\lambda(q_k)}.
$$
\end{corollary}
\begin{proof}
Notice that, since $G$ is walk-regular, $(\A^{\ell})_{uu}=\frac{1}{n}\tr \A^{\ell}$ for any $u\in V$ and $\ell=0,1,\ldots,k$. Thus,
$$
W(q_k)=\frac{1}{n}\tr q_k(\A)=\frac{1}{n}\sum_{i=0}^d m_i q_k(\theta_i)=\langle q_k,1\rangle_G=\|p_0\|^2=1,
$$
where we used that $p_0=1$.
Moreover, from the orthogonality of the polynomials $q_i$, and $q_0=1$,
$\langle 1,q_k\rangle_{[G]}=\frac{1}{n}\sum_{i=0}^{d-1}m_i(\theta_0-\theta_i)q_k(\theta_i)=0$. Then, since $q_k(\theta_0)>0$, it must be $q_k(\theta_i)<0$ for some $i>0$, and hence $q_k(\theta_0)>\lambda(q_k)$, as required. Then, the result follows from Theorem \ref{theo2}.
\end{proof}

\subsection*{An Example}

To compare the above bounds with those  obtained in \cite{Fiolkindep} and \cite{act16} (here in Theorems \ref{thm:fiol} and \ref{thm:abiad}, respectively), let us consider $G$ to be the Johnson graph $J(14,7)$ (see, for instance, \cite{bcn89,g93}). This is an antipodal (but not bipartite) distance-regular graph, with $n=3432$ vertices, diameter $D=7$, and spectrum
$$
\spec G=\{49^1, 35^{13}, 23^{77}, 13^{273}, 5^{637}, -1^{1001}, -5^{1001}, -7^{429}\}.
$$
In Table \ref{table2} we show the bounds obtained for $\alpha_k$, together with the values of $P_k(\theta_0)$, $W_k$, $\theta$,  $\lambda(p)$, $q_k(\delta)$, and  $\lambda(q_k)$, for $k=3,\ldots,7$. Since every distance-regular graph is also walk-regular, the value of $W_k$ is just $\frac{1}{n}\tr p(\A)$, easily computed from the spectrum. Note that, in general, the bounds obtained by the above corollaries constitute a significant improvement with respect to those in \cite{Fiolkindep,act16}. In particular, the bounds for $k=6,7$ are either equal or quite closed to the correct values $\alpha_6=2$ (since $G$ is $2$-antipodal, and $\alpha_7=1$ (since $D=7$).

\begin{table}
\begin{center}
\begin{tabular}{|c|ccccc| }
\hline
$k$ & $3$  & $4$ & $5$ & $6$ & $7$ \\
\hline
$P_k(\theta_0)$ & 464 & 125 & 20 & 2 & -- \\
\hline
$W_k$ ($p=x+\cdots+x^k$) & 637 & 17150 & 469910 & 15193479 & 537790827 \\
\hline
$\theta$  & $35$  & $35$ & $35$ & $35$ & $35$ \\
\hline
$\lambda(p)$ ($p=x+\cdots+x^k$) & -301 & 0 & -14707 & 0 & -720601 \\
\hline
$q_k(\delta)$ & 1716 & 2941 & 3382 & 3431 & 3432 \\
\hline
$\lambda(q_k)$ & -40 & -75 & -24 & -1 & 0 \\
\hline
Bound from Theorem \ref{thm:fiol} & 464 & 125 & 20 & 2  & -- \\
\hline
Bound from Theorem \ref{thm:abiad} & 935 & 721 & 546 & 408  & 302 \\
\hline
Bound from Corollary \ref{coro-gen-k} & 26 & 10 & 5 & 3  & 2 \\
\hline
Bound from Corollary \ref{coro-walk-reg} & 80 & 86 & 25 & 2  & 1 \\
\hline
\end{tabular}
\end{center}
\caption{Comparison of bounds for $\alpha_k$ in the Johnson graph $J(14,7)$.}
\label{table2}
\end{table}

\subsection{An antipodal-like bound}

\begin{theorem}
\label{theo3}
Let $G$ be a regular graph with a maximum $k$-independent set of size $r$. Let  $p\in \Re_k[x]$ be a polynomial satisfying $p(\lambda_1)\ge \Lambda(p)> 0$, $\lambda(p)<0$, and assume that $\Lambda(p) \geq|\lambda(p)|(r-1)$.
Then,
\begin{equation}
\label{antipod-bound}
r = \alpha_k\le  1+\frac{\Lambda(p)}{p(\lambda_1)}(n-1).
\end{equation}
\end{theorem}
\begin{proof}
Let $U=\{u_0,u_1,\ldots,u_{r-1}\}$ be a maximum
$k$-independent set, where $r=|U|=\alpha_k$.  The  matrix
$p(\A)$ has eigenvalues
$p(\lambda_1)\ge \Lambda(p)$ and $p(\lambda_i)$ satisfying  $\lambda(p)\le
p(\lambda_i)\le \Lambda(p)$ for
$2\le i\le n$. Now consider the matrix $\B:=\A(K_r)\otimes p(\A)$.
For instance, for $r=3$ we have
$$
\B=
\left(
\renewcommand{\arraystretch}{1.8}
\begin{array}{c|c|c}
\O &  p(\A) & p(\A) \\
\hline
p(\A) &   \O &  p(\A) \\
\hline
p(\A) &  p(\A)  &    \O
\end{array}
\right).
$$

The complete graph $K_r$ has  eigenvalues $r-1$,  and $-1$ with
multiplicity $r-1$, with corresponding orthogonal eigenvectors
$\vecphi_0 = \j\in \R^r$ and
$\vecphi_i=(1,\omega^i,\omega^{2i},\ldots,\omega^{(r-1)i})^{\top}$,
$1\le i\le r-1$, where
$\omega$ is a primitive $r$-th root of unity, say
$\omega:=e^{j\frac{2\pi}{r}}$. Consequently, each  eigenvector
$\vecu$ of $p(\A)$, with eigenvalue $p(\lambda)$, $\lambda\in
\spec G$, gives rise to the eigenvalues $(r-1)p(\lambda)$, and
$-p(\lambda)$ with multiplicity $r-1$, with corresponding
orthogonal eigenvectors $\vecu_0:=\j\otimes \vecu$ and
$\vecu_i:=\vecphi_i\otimes \vecu$, $1\le i\le r-1$.
Thus, when $\lambda\neq \lambda_1$, we know that
$\lambda(p)\le p(\lambda)\le \Lambda(p)$ and, hence, the corresponding eigenvalues
of $\B$ are within the interval $[\lambda(p)(r-1), \Lambda(p)(r-1)]$. Moreover, $\B$ has
maximum eigenvalue $(r-1)p(\lambda_1)\ge \Lambda(p)(r-1)$.

Now consider the (column) vector
$\f_U:=(\e_{u_0}^\top |\e_{u_1}^\top | \cdots | \e_{u_{r-1}}^\top
)^\top\in \R^{rn}$, and consider its spectral decomposition:
\begin{equation}
\label{spectral-decomp}
\f_U= \sum_{i=0}^{r-1}\frac{\langle  \f_U, \j_i\rangle
}{\|\j_i\|^2}\j_i+\z_U=\frac{1}{n}\j_0+\z_U
\end{equation}
where $\j_i=\vecphi_i\otimes \j$, $\z_U\in \langle\j_0,\j_1,\ldots,
\j_{r-1}\rangle ^{\bot}$, and we have used that
$\langle  \f_U, \j_0\rangle=r$, $\|\j_0\|^2=rn$, and $\langle  \f_U,
\j_i\rangle=\sum_{j=0}^{r-1}\omega^{ij}=0$, for any
$1\le i\le r-1$. From (\ref{spectral-decomp}), we get
$$
\|\z_U \|^2
=\|\f_U\|^2-\frac{1}{n^2}\|\j_0\|^2=r\left(1-\frac{1}{n}\right).
$$

Since there is no path of length $\le k$ between any pair of vertices
of $U$,
$(p(\A))_{u_iu_j}=0$ for any $i\neq j$. Thus,
\begin{eqnarray*}
0 & = & \langle \B\f_U, \f_U\rangle =
\left\langle
\frac{(r-1)p(\lambda_1)}{n}\j_0+\B\z_U, \frac{1}{n}\j_0+\z_U
\right\rangle \\ & = & \frac{r(r-1)p(\lambda_1)}{n} +\langle
\B\z_U, \z_U\rangle \\ &
\ge &  \frac{r(r-1)p(\lambda_1)}{n} +
(r-1)\lambda(p)\|\z_U\|^2  \\
 & = & \frac{r(r-1)}{n}p(\lambda_1) + (r-1)\lambda(p)r\left(\frac{n-1}{n}\right).
\end{eqnarray*} Therefore, we get
$$
p(\lambda_1)\le -\lambda(p)(n-1)=|\lambda(p)|(n-1)=\frac{\Lambda(p)}{r-1}(n-1),
$$
whence \eqref{antipod-bound} follows.
\end{proof}

As a consequence of the above theorem, we have the following result.

\begin{corollary}
\label{coro-antipod-P}
Let $P\in\R_k[x]$ satisfying $P(\lambda_1)\ge \Lambda(P)$. Then,
\begin{equation}
\label{bound-antipod-P}
\alpha_k\le \frac{n(\Lambda(P)-\lambda(P))}{P(\lambda_1)-\lambda(P)}.
\end{equation}
\end{corollary}
\begin{proof}
Notice that, if $P\in\R_k[x]$ is a polynomial with $P(\lambda_1)\ge \Lambda(P)$, and $r >1$, then the polynomial
$$
p(x)=\frac{r}{\Lambda(P)-\lambda(P)}P(x) - \frac{r\lambda(P)}{\Lambda(P)-\lambda(P)}-1
$$
satisfies the hypotheses of Theorem \ref{theo3}. Then, using $p(x)$ with  $r=\alpha_k$ in \eqref{antipod-bound} and solving for $\alpha_k$ we obtain the desired result.
\end{proof}

Note that if $P=P_k$, the $k$-alternating polynomial, Corollary \ref{coro-antipod-P} yields Theorem \ref{thm:fiol}. In particular, in \cite{Fiolkindep} it was shown that
the bound \eqref{bound-antipod-P}  for $\alpha _{d-}$ and $P=P_{d-1}$ is attained for every $r$-antipodal distance-regular graph with $d+1$ distinct eigenvalues (see \cite{Fiolkindep}). For example, one can easily check that, with the order and eigenvalues of $L(K_{2,k+1})$ in \eqref{eigenvals-LK(2,k+1)}, the bound in Theorem \ref{thm:fiol} gives the right value $\alpha_2=2$.

\section{Bounding the diameter}

As a by-product of our results, we can also obtain upper bounds for the diameter of a graph $G$. This is because if $\alpha_k=1$ for some $k$, then the diameter of $G$ must satisfy $D\le k$. To assure that $\alpha_k = 1$, we only need to obtain an upper bound smaller than $2$. As an example, the following result follows as a direct consequence of Corollary \ref{k=2}.

\begin{proposition}
\label{diam}
Let $G$ be a regular graph on $n$ vertices, and with distinct eigenvalues $\theta_0>\cdots>\theta_d$, $d\ge 2$. Let $\theta_i$ the largest eigenvalue not greater than $-1$. Then, if
$$
n\frac{\theta_0+\theta_i\theta_{i-1}}
{(\theta_0-\theta_i)(\theta_0-\theta_{i-1})}<2
$$
then $G$ has diameter $D=2$.
\end{proposition}

Another interesting conclusion is a result which was first obtained in \cite{fgy96}. Here we show it as a consequence of Corollary \ref{coro-antipod-P} by taking $P=P_k$, the $k$-alternating polynomial (that is, Theorem \ref{thm:fiol}).

\begin{proposition}[\cite{fgy96}]
Let $G$ be a regular graph with $n$ vertices, diameter $D$, and distinct eigenvalues $\theta_0>\cdots > \theta_d$. For some $k\le d-1$, Let $P_k$ be the corresponding $k$-alternating polynomial.
If $P_k(\theta_0)> n-1$, then $D\le k$.
\end{proposition}
\begin{proof}
The sufficient condition comes from assuming that $\alpha_k\le \frac{2n}{P_k(\theta_0)+1}<2$.
\end{proof}


\subsection*{Acknowledgments}
We would like to thank S. M. Cioab\u{a} for helpful discussions during an early stage of this work.\\
We greatly acknowledge the anonymous referees for their detailed  comments, suggestions, and queries. They led to significant improvements on the first version of this article. In particular, both referees encouraged us to get an improvement on Corollary \ref{k=2}, by asking for the right choice of the polynomials involved. In this context, the first referee gave some particular examples, whereas the second one pointed out to the best polynomials for many cases. This led to significant improvements on many entries in Table \ref{table1}, and also led us to find the infinite families of distance-regular graphs where our new bound is tight.\\
Research of M. A. Fiol was partially supported by the project 2017SGR1087 of the Agency for the Management of University and Research Grants (AGAUR) of the Government of Catalonia. G. Coutinho acknowledges a travel grant from the Dept. of Computer Science of UFMG.


\end{document}